\documentclass[a4paper,12pt]{article}
\usepackage{latexsym,amssymb}
\usepackage{amsmath}
\usepackage[mathscr]{eucal}
\usepackage{epic,eepic}
\usepackage{color}
\usepackage{enumerate}

\usepackage[anchorcolor=blue,%
 bookmarks=true,%
 bookmarksnumbered=true,%
 colorlinks=true,%
 citecolor=cyan,%
 linkcolor=blue,%
 dvipdfmx%
]{hyperref}
\definecolor{refkey}{gray}{0.5}
\definecolor{labelkey}{gray}{0.2}

\usepackage{amsthm}
\newtheorem{theorem}{Theorem}[section]
\newtheorem{proposition}[theorem]{Proposition}
\newtheorem{lemma}[theorem]{Lemma}
\newtheorem{corollary}[theorem]{Corollary}

\theoremstyle{definition}
\newtheorem{definition}[theorem]{Definition}

\theoremstyle{remark}
\newtheorem{remark}[theorem]{Remark}
\newtheorem{example}[theorem]{Example}

\makeatletter

\@addtoreset{equation}{section}
\makeatother

\newcommand{\R}{\mathbb{R}}
\newcommand{\sJ}{\mathsf{J}}

\newcommand{\ve}{\varepsilon}

\newcommand{\lra}{\longrightarrow}
\newcommand{\e}{\mathrm{e}}

\newcommand{\CAT}{\mathrm{CAT}}

\newcommand{\RCD}{\mathrm{RCD}}

\newcommand{\diam}{\mathrm{diam}}

\def\argmin{\mathop{\mathrm{arg\, min}}}

\begin{document}

\title{Discrete-time gradient flows in Gromov hyperbolic spaces}

\author{Shin-ichi OHTA\thanks{
Department of Mathematics, Osaka University, Osaka 560-0043, Japan
({\sf s.ohta@math.sci.osaka-u.ac.jp}),
RIKEN Center for Advanced Intelligence Project (AIP),
1-4-1 Nihonbashi, Tokyo 103-0027, Japan}}

%\date{\today}
\date{\empty}
\maketitle

\begin{abstract}
We investigate fundamental properties of the proximal point algorithm
for Lipschitz convex functions on (proper, geodesic) Gromov hyperbolic spaces.
We show that the proximal point algorithm from an arbitrary initial point
can find a point close to a minimizer of the function.
Moreover, we establish contraction estimates (akin to trees)
for the proximal (resolvent) operator.
Our results can be applied to small perturbations of trees.
\end{abstract}

%%%%%%%%%%%%%%%%%%%%%%%%%%%%%%%%%%%%

\section{Introduction}%%%%%
%%%%%%%%%%%%%%%%%%%%%%%%%%%%%%%

This article is devoted to an attempt to develop optimization theory
on ``non-Riemannian'' metric spaces.
Precisely, we study the discrete-time gradient flow for a convex function $f$ on a metric space $(X,d)$
built of the \emph{proximal} (or \emph{resolvent}) \emph{operator}
\begin{equation}\label{eq:Jf}
\sJ^f_{\tau}(x) :=\argmin_{y \in X} \bigg\{ f(y) +\frac{d^2(x,y)}{2\tau} \bigg\},
\end{equation}
where $\tau>0$ is the step size.
Iterating $\sJ^f_{\tau}$ is a well known scheme to construct a continuous-time gradient flow for $f$
in the limit as $\tau \to 0$ (we refer to \cite{Br} for the classical setting of Hilbert spaces
and to \cite{CL,GR,Mi} for some related works).
Generalizations of the theory of gradient flows to convex functions on metric spaces
go back to 1990s \cite{Jo,Jo2,Ma} and have been making impressive progress since then;
we refer to \cite{AGSbook,Ba,Babook} (to name a few) for the case of $\CAT(0)$-spaces,
\cite{OP1,OP2} for $\CAT(1)$-spaces,
\cite{Ly,Ogra,OP1,Pe,Sa} for Alexandrov spaces and the Wasserstein spaces over them,
and to \cite{St} for metric measure spaces satisfying the Riemannian curvature-dimension condition
($\RCD(K,\infty)$-spaces for short).
Here a \emph{$\CAT(k)$-space} (resp.\ an \emph{Alexandrov space} of curvature $\ge k$)
is a metric space with sectional curvature bounded from above (resp.\ below) by $k \in \R$,
and an \emph{$\RCD(K,\infty)$-space} is a metric measure space
of Ricci curvature bounded from below by $K \in \R$, in certain synthetic geometric senses.
These spaces are all ``Riemannian'' in the sense that
non-Riemannian Finsler manifolds are excluded.

The theory of gradient flows in $\CAT(0)$-spaces has found applications in optimization theory.
Some important classes of spaces turned out $\CAT(0)$-spaces
(such as phylogenetic tree spaces \cite{BHV}
and the orthoscheme complexes of modular lattices \cite[Chapter~7]{CCHO}; see also \cite{Babook}),
and then optimization in $\CAT(0)$-spaces
can be applied to solve problems in optimization theory (see, e.g., \cite{HH,Hi}).

Compared with the development of the theory of gradient flows in Riemannian spaces as above,
we know much less for non-Riemannian spaces (even for normed spaces).
Especially, the lack of the contraction (non-expansion) property is a central problem.
The aim of this article is to contribute to closing this gap.
For this purpose, we consider discrete-time gradient flows for convex functions
on Gromov hyperbolic spaces.

The \emph{Gromov hyperbolicity}, introduced in a seminal work \cite{Gr} of Gromov,
is a notion of negative curvature of large scale.
A metric space $(X,d)$ is said to be Gromov hyperbolic
if it is \emph{$\delta$-hyperbolic} for some $\delta \ge 0$ in the sense that
\begin{equation}\label{eq:Ghyp}
(x|z)_p \ge \min\{ (x|y)_p, (y|z)_p \} -\delta
\end{equation}
holds for all $p,x,y,z \in X$, where
\[ (x|y)_p:=\frac{1}{2} \big\{ d(p,x)+d(p,y)-d(x,y) \big\} \]
is the \emph{Gromov product}.
If \eqref{eq:Ghyp} holds with $\delta=0$,
then the quadruple $p,x,y,z$ is isometrically embedded into a tree.
Therefore, the $\delta$-hyperbolicity means that $(X,d)$ is close to a tree
up to local perturbations of size $\delta$ (cf.\ Example~\ref{ex:Ghyp}(e)).
Admitting such a local perturbation is a characteristic feature of the Gromov hyperbolicity;
this is a reason why the Gromov hyperbolicity plays a vital role in group theory
and some non-Riemannian Finsler manifolds (e.g., Hilbert geometry) can be Gromov hyperbolic
(see Example~\ref{ex:Ghyp} for a further account).
We refer to \cite{BCCM,CCDDMV,CCL,FIV} and the references therein
for some investigations on the computation of $\delta$.

Inspired by the success of the theory of gradient flows in $\CAT(0)$-spaces,
it is natural to consider gradient flows in Gromov hyperbolic spaces
(note that trees are $\CAT(k)$ for any $k \in \R$),
and then we should employ discrete-time gradient flows
because of the inevitable local perturbations.
Precisely, for a convex function $f:X \lra \R$,
we study the behavior of the proximal operator $\sJ^f_{\tau}$ as in \eqref{eq:Jf}.
Then, due to the possible local perturbations of size $\delta$,
only $\sJ^f_{\tau}$ for large $\tau$ (``giant steps'') is meaningful (see Example~\ref{ex:Ghyp}(c),
from which we find that any nontrivial estimate on the local behavior cannot be expected).
We remark that $\sJ^f_{\tau}(x) \neq \emptyset$ under a mild compactness assumption
(see the beginning of Subsection~\ref{ssc:thm1}).

Our first main result is the following (see \eqref{eq:Kconv} for the definition of the $K$-convexity).

\begin{theorem}[Tendency towards minimizer]\label{th:conv}
Let $(X,d)$ be a proper $\delta$-hyperbolic geodesic space,
and $f:X \lra \R$ be a $K$-convex $L$-Lipschitz function with $K \ge 0$, $L>0$
such that $\inf_X f$ is attained at some $p \in X$.
Then, for any $x \in X$, $\tau>0$, and $y \in \sJ^f_{\tau}(x)$, we have
\begin{equation}\label{eq:conv}
d(p,y) \le d(p,x) -d(x,y) +\frac{4\sqrt{2\tau L\delta}}{\sqrt{K\tau +1}}.
\end{equation}
In the case of $K>0$ and $\tau>K^{-1}$, we further obtain
\begin{equation}\label{eq:conv-d}
d(p,y) \le d(p,x)
 -\bigg( 1-\frac{1}{K\tau} \bigg) \frac{f(x)-f(p)}{L} +\frac{4\sqrt{2\tau L\delta}}{\sqrt{K\tau +1}}.
\end{equation}
\end{theorem}

The inequality \eqref{eq:conv-d} ensures that,
if $f(x)$ is sufficiently larger than $f(p)=\inf_X f$ (relative to $\delta$),
then the operator $\sJ^f_{\tau}$ sends $x$ to a point closer to $p$.

We remark that, in the case of $K>0$ and $\tau >K^{-1}$,
the $K$-convexity and the $L$-Lipschitz continuity imply
\begin{equation}\label{eq:conv-f}
f(y) \le f(x) -\frac{(K\tau -1)^2(f(x)-f(p))^2}{2(KL)^2 \tau^3}
\end{equation}
regardless of the $\delta$-hyperbolicity.
Then, given $\ve>0$ and an arbitrary initial point $x_0 \in X$,
by recursively choosing $x_i \in \sJ^f_{\tau}(x_{i-1})$, we have
\begin{equation}\label{eq:fxN}
f(x_N) \le f(p) +\frac{KL\tau \sqrt{2\tau}}{K\tau -1} \ve
\end{equation}
for some $N <(f(x_0)-f(p))\ve^{-2}$.
Together with the $K$-convexity, \eqref{eq:fxN} yields
\begin{equation}\label{eq:dpxN}
d^2(p,x_N) \le \frac{2L\tau \sqrt{2\tau}}{K\tau -1} \ve
\end{equation}
(see Subsection~\ref{ssc:cor} for more details).
By a similar discussion based on \eqref{eq:conv-d},
we obtain the following estimate in $\delta$-hyperbolic spaces.

\begin{corollary}\label{cr:conv}
Let $(X,d)$ and $f$ be as in Theorem~$\ref{th:conv}$ with $K>0$ and $\tau>K^{-1}$.
Then, given $\ve>0$ and an arbitrary initial point $x_0 \in X$, we have
\begin{equation}\label{eq:dpxN+}
d^2(p,x_N) \le \frac{2L\tau}{K\tau -1} \bigg( \frac{4\sqrt{2\tau L\delta}}{\sqrt{K\tau +1}} +\ve^2 \bigg)
\end{equation}
for some $N <d(p,x_0)\ve^{-2}$.
\end{corollary}

Note that, up to a constant depending on $\delta$,
the order $\ve^2$ in \eqref{eq:dpxN+} is better than $\ve$ in \eqref{eq:dpxN}.
We refer to \cite{Ba,OP1} for the convergence of discrete-time gradient flows
(i.e., $x_N$ converges to a minimizer of $f$) in metric spaces with upper or lower
sectional curvature bounds.

Our second main result establishes the contraction property of the proximal operator.

\begin{theorem}[Contraction estimates]\label{th:cont}
Let $(X,d)$ and $f$ be as in Theorem~$\ref{th:conv}$.
Take any $x_1,x_2 \in X$, $\tau>0$, and $y_i \in \sJ^f_{\tau}(x_i)$ for $i=1,2$,
and assume $d(p,y_1) \le d(p,y_2)$.
\begin{enumerate}[{\rm (i)}]
\item
If $d(p,y_1) \ge (x_1|x_2)_p$, then we have
\begin{equation}\label{eq:cont1}
d(y_1,y_2) \le d(x_1,x_2) -d(x_1,y_1) -d(x_2,y_2)
 +\frac{20\sqrt{2\tau L\delta}}{\sqrt{K\tau +1}} +24\delta.
\end{equation}
In the case of $K>0$ and $\tau >K^{-1}$, we further obtain
\begin{equation}\label{eq:cont1+}
d(y_1,y_2) \le d(x_1,x_2) -\bigg( 1-\frac{1}{K\tau} \bigg) \frac{f(x_1)+f(x_2)-2f(p)}{L}
 +\frac{20\sqrt{2\tau L\delta}}{\sqrt{K\tau +1}} +24\delta.
\end{equation}

\item
If $d(p,y_1) <(x_1|x_2)_p$, then we have
\begin{equation}\label{eq:cont2}
d(y_1,y_2) \le d(x_1,x_2) -(p|x_2)_{x_1} +C(K,L,D,\tau,\delta),
\end{equation}
where $D:=\max\{d(p,x_1),d(p,x_2)\}$ and
$C(K,L,D,\tau,\delta)=O_{K,L,D,\tau}(\delta^{1/4})$ as $\delta \to 0$.
\end{enumerate}
\end{theorem}

See Subsection~\ref{ssc:thm2} for a precise estimate of $C(K,L,D,\tau,\delta)$.
The inequalities \eqref{eq:conv}, \eqref{eq:cont1} and \eqref{eq:cont2} show that
$\sJ^f_{\tau}$ behaves like that in a tree (see Subsection~\ref{ssc:PPA})
up to a difference depending on $\delta$.
Note also that \eqref{eq:conv} can be regarded as a contraction estimate
between $p$ and $x \longmapsto y$.

For gradient curves $\gamma,\eta$ of a $K$-convex function on a Riemannian space,
the \emph{exponential contraction}
\[ d\big( \gamma(t),\eta(t) \big) \le \e^{-Kt} d\big( \gamma(0),\eta(0) \big) \]
is known as one of the most important properties and has a number of applications
from the uniqueness of gradient curves to the analysis of heat flow (see, e.g., \cite{AGSbook}).
For example, the exponential contraction of heat flow plays a significant role
in geometric analysis on $\RCD(K,\infty)$-spaces;
heat flow can be regarded as gradient flow of the relative entropy
in the $L^2$-Wasserstein space, and the $K$-convexity of the relative entropty
is exactly the definition of the curvature-dimension condition
(we refer to \cite{AGS,EKS,Vi}).
For non-Riemannian spaces (such as normed spaces and Finsler manifolds), however,
the exponential contraction is known to fail (see \cite{OSnc}).
To the best of the author's knowledge,
Theorem~\ref{th:cont} is the first contraction estimate
concerning gradient flows of convex functions on non-Riemannian spaces.

This article is organized as follows.
We briefly review the basics of Gromov hyperbolic spaces
and the proximal point algorithm in Section~\ref{sc:pre}.
Then Section~\ref{sc:proof} is devoted to the proofs of the main results
and discussions on possible further investigations.

\section{Preliminaries}\label{sc:pre}%%%%%
%%%%%%%%%%%%%%%%%%%%%%%%%%%%%%%

For $a,b \in \R$, we set $a \wedge b:=\min\{a,b\}$ and $a \vee b:=\max\{a,b\}$.
Besides the original paper \cite{Gr}, we refer to \cite{Bo,BH,DSU,Ro,Va}
for the basics and various applications of Gromov hyperbolic spaces.

\subsection{Gromov hyperbolic spaces}\label{ssc:Ghyp}%%%%%
%%%%%%%%%%%%%%%%%%%%%%%%%%%%%%%

We first have a closer look on the Gromov hyperbolicity mentioned in the introduction.
Let $(X,d)$ be a metric space.
For three points $x,y,z \in X$, define the \emph{Gromov product} $(y|z)_x$ by
\[ (y|z)_x :=\frac{1}{2} \big\{ d(x,y) +d(x,z) -d(y,z) \big\}. \]
Observe from the triangle inequality that
\begin{equation}\label{eq:Gpro-1}
0 \le (y|z)_x \le d(x,y) \wedge d(x,z).
%|(y|z)_x -(y'|z)_x| \le d(y,y'). \label{eq:Gpro-2}
\end{equation}
In the Euclidean plane $\R^2$,
$(y|z)_x$ is understood as the distance from $x$
to the intersection of the triangle $\triangle xyz$ and its inscribed circle
(see the left triangle in Figure~\ref{fig:Gprod}).
If $x,y,z$ are in a tripod, then $(y|z)_x$ coincides with the distance from $x$ to the branching point
(see the right figure in Figure~\ref{fig:Gprod}).

%%%%%%%%%%%%%%%%%%%%%%%%
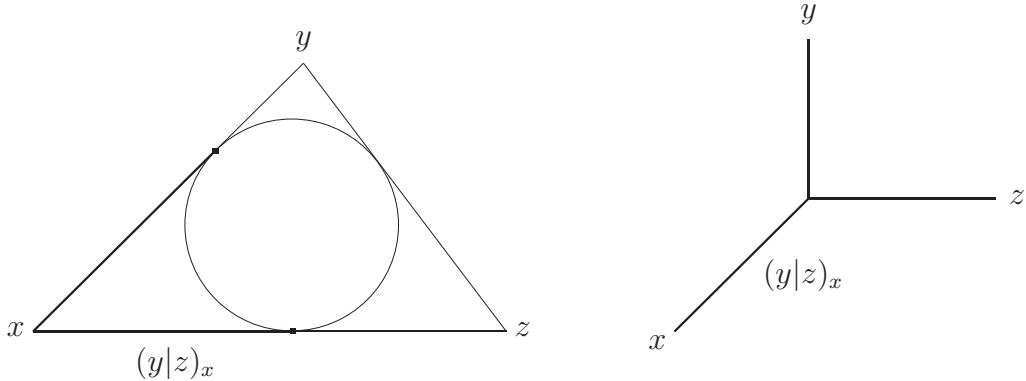
\begin{figure}
\centering
\begin{picture}(400,150)

\put(10,20){\line(1,0){177}}
\put(10,20){\line(1,1){101}}
\put(111,121){\line(3,-4){76}}
\put(106.5,60){\circle{80}}

\put(0,18){$x$}
\put(108,128){$y$}
\put(190,18){$z$}

\put(106,19){\rule{2pt}{2pt}}
\put(77,87){\rule{2pt}{2pt}}
\put(48,5){$(y|z)_x$}

\thicklines
\put(10,20){\line(1,0){96}}
\put(10,20){\line(1,1){68}}
%%%%%%%%%%%%%%%%%%%%%%%%

\put(250,20){\line(1,1){50}}

\thinlines
\put(300,70){\line(0,1){60}}
\put(300,70){\line(1,0){70}}

\put(240,13){$x$}
\put(297,138){$y$}
\put(375,68){$z$}
\put(283,38){$(y|z)_x$}

\end{picture}
\caption{Gromov products in $\R^2$ and a tripod}\label{fig:Gprod}
\end{figure}
%%%%%%%%%%%%%%%%%%%%%%%%

\begin{definition}[Gromov hyperbolic spaces]\label{df:Ghyp}
A metric space $(X,d)$ is said to be \emph{$\delta$-hyperbolic} for $\delta \ge 0$ if
\begin{equation}\label{eq:d-hyp}
(x|z)_p \ge (x|y)_p \wedge (y|z)_p -\delta
\end{equation}
holds for all $p,x,y,z \in X$.
We say that $(X,d)$ is \emph{Gromov hyperbolic}
if it is $\delta$-hyperbolic for some $\delta \ge 0$.
\end{definition}

The Gromov hyperbolicity can be regarded as a large-scale notion of negative curvature.

\begin{example}\label{ex:Ghyp}
\begin{enumerate}[(a)]
\item
Complete, simply connected Riemannian manifolds of sectional curvature $\le -1$
(or, more generally, $\CAT(-1)$-spaces)
are Gromov hyperbolic (see \cite[Proposition~H.1.2]{BH}).

\item
An important difference between $\CAT(-1)$-spaces and Gromov hyperbolic spaces
is that the latter admits some non-Riemannian Finsler manifolds such as Hilbert geometry
(see \cite{KN}, \cite[\S 6.5]{Obook}).
We also remark that, for the Teichm\"uller space of a surface of genus $g$ with $p$ punctures,
the Weil--Petersson metric (which is Riemannian and incomplete)
is known to be Gromov hyperbolic if and only if $3g-3+p \le 2$ (\cite{BF}),
whereas the Teichm\"uller metric (which is Finsler and complete)
does not satisfy the Gromov hyperbolicity (\cite{MW})
(see also \cite[\S 6.6]{Obook}).

\item
It is clear from \eqref{eq:Gpro-1} that the Gromov product
does not exceed the diameter $\diam(X):=\sup_{x,y \in X}d(x,y)$.
Hence, if $\diam(X) \le \delta$, then $(X,d)$ is $\delta$-hyperbolic.
This also means that the local structure of $X$ (up to size $\delta$)
is not influential in the $\delta$-hyperbolicity.

\item
The definition \eqref{eq:d-hyp} makes sense for discrete spaces.
In fact, the Gromov hyperbolicity has found rich applications in group theory
(a discrete group whose Cayley graph satisfies the Gromov hyperbolicity
is called a \emph{hyperbolic group}; we refer to \cite{Bo,Gr}, \cite[Part~III]{BH}).
In the sequel, however, we do not consider discrete spaces,
mainly due to the difficulty of dealing with convex functions (see Subsection~\ref{ssc:prob}).

\item
Assume that $(X,d)$ admits a map $\phi:T \lra X$ from a tree $(T,d_T)$ such that
$d(\phi(a),\phi(b))=d_T(a,b)$ for all $a,b \in T$
and that the $\delta$-neighborhood $B(\phi(T),\delta)$ of $\phi(T)$ covers $X$.
Then, since $(T,d_T)$ is $0$-hyperbolic, we can easily see that $(X,d)$ is $6\delta$-hyperbolic.
\end{enumerate}
\end{example}

We call $(X,d)$ a \emph{geodesic space}
if any two points $x,y \in X$ are connected by a (minimal) \emph{geodesic}
$\gamma:[0,\ell] \lra X$ satisfying $\gamma(0)=x$, $\gamma(\ell)=y$ and
$d(\gamma(s),\gamma(t))=(|s-t|/\ell) \cdot d(x,y)$ for all $s,t \in [0,\ell]$
(we will take $\ell=1$ or $\ell=d(x,y)$).
In this case, there are many characterizations of the Gromov hyperbolicity,
most notably by the \emph{$\delta$-slimness} of geodesic triangles (see, e.g., \cite[\S III.H.1]{BH}).
We remark that, by \cite[Theorem~4.1]{BS},
every $\delta$-hyperbolic metric space can be isometrically embedded
into a complete $\delta$-hyperbolic geodesic space.
Concerning the Gromov product in a $\delta$-hyperbolic geodesic space,
one can see that
\begin{equation}\label{eq:dxg}
d(x,\gamma) -2\delta \le (y|z)_x \le d(x,\gamma),
\end{equation}
where $d(x,\gamma) :=\min_{t \in [0,1]} d(x,\gamma(t))$,
holds for any $x,y,z \in X$ and geodesic $\gamma:[0,1] \lra X$ from $y$ to $z$
(note that the latter inequality always holds by the triangle inequality; see \cite[2.33]{Va}).

We close this subsection with two important fundamental lemmas
for later use in the proofs of Theorems~\ref{th:conv} and \ref{th:cont}, respectively
(see \cite[2.15, 2.19]{Va}).
%We give proofs for completeness since they are elementary

\begin{lemma}[Tripod lemma]\label{lm:tripod}
Let $\gamma,\eta:[0,1] \lra X$ be geodesics emanating from the same point $x$
and put $y=\gamma(1)$, $z=\eta(1)$.
Then, for any $y'$ on $\gamma$ and $z'$ on $\eta$ with $d(x,y')=d(x,z') \le (y|z)_x$, we have
\[ d(y',z') \le 4\delta. \]
\end{lemma}

\if0%%%%%
\begin{proof}
Iterating \eqref{eq:d-hyp}, we have
\begin{align*}
(y'|z')_x
&\ge (y'|y)_x \wedge (y|z')_x -\delta \\
&\ge  (y'|y \big)_x \wedge (y|z)_x \wedge (z|z')_x -2\delta \\
&= d(x,y') \wedge (y|z)_x \wedge d(x,z') -2\delta.
\end{align*}
Then, by the hypothesis $d(x,y')=d(x,z') \le (y|z)_x$, we obtain the claim.
%$\qedd$
\end{proof}
\fi%%%%%

\begin{lemma}\label{lm:trian}
Let $\gamma_i$ be a geodesic from $p$ to $x_i$, $i=1,2$.
Then, for $y_i$ on $\gamma_i$ such that $d(p,y_1) \wedge d(p,y_2) \ge (x_1|x_2)_p -\sigma$
with $\sigma \ge 0$, we have
\[ |(x_1|x_2)_p -(y_1|y_2)_p| \le 6\delta +\sigma. \]
\end{lemma}

In view of \eqref{eq:dxg}, the latter lemma means that
the distance from $p$ to a geodesic between $x_1$ and $x_2$
is almost the same as the distance from $p$ to a geodesic between $y_1$ and $y_2$.

%%%%%

\subsection{Proximal point algorithm}\label{ssc:PPA}%%%%%
%%%%%%%%%%%%%%%%%%%%%%%%%%%%%%%

Given a function $f:X \lra \R$ on a metric space $(X,d)$,
optimization theory is concerned with how to find a minimizer (or the minimum value) of $f$.
It is well studied for $\CAT(0)$-spaces by means of the \emph{proximal point algorithm};
we refer to the books \cite{AGSbook,Babook} for further reading.
For $x \in X$ and $\tau>0$,
recall that the \emph{proximal}  (or \emph{resolvent}) \emph{operator} is defined by
\[ \sJ^f_{\tau}(x) :=\argmin_{y \in X} \bigg\{ f(y) +\frac{d^2(x,y)}{2\tau} \bigg\}. \]
Roughly speaking, an element in $\sJ^f_{\tau}(x)$ can be regarded as
an approximation of a point on the gradient curve of $f$ at time $\tau$ from $x$.

As a fundamental example,
let us consider a convex function on a $0$-hyperbolic geodesic space.
We say that $f$ is (weakly, geodesically) \emph{$K$-convex} for $K \in \R$ if, for any $x,y \in X$
and some geodesic $\gamma:[0,1] \lra X$ from $x$ to $y$,
\begin{equation}\label{eq:Kconv}
f\big( \gamma(t) \big) \le (1-t)f(x) +tf(y) -\frac{K}{2}(1-t)td^2(x,y)
\end{equation}
holds for all $t \in [0,1]$.
As usual, by a convex function we mean a $0$-convex function.

Let $(X,d)$ be a $0$-hyperbolic geodesic space
and $f$ be a convex function on $X$ such that $\inf_X f$ is attained at $p \in X$.
By the $0$-hyperbolicity, any four points in $X$ are isometrically embedded into a tree and,
in particular, any two points are connected by a unique geodesic
(see, e.g., \cite[\S 3.3]{DSU}, \cite[\S 6.2]{Ro}).
Given $x \in X$ and $\tau>0$,
we take $y \in \sJ^f_{\tau}(x)$ and assume $f(y)>f(p)$.
Then, on the geodesic $\gamma:[0,1] \lra X$ from $x$ to $y$,
we find from the choice of $y$ that $f(y)<f(\gamma(t))$ holds for all $t \in [0,1)$.
Let $\gamma(\bar{t})$ be the closest point to $p$ on $\gamma$.
Then the concatenation of the geodesic $\eta$ from $p$ to $\gamma(\bar{t})$
and $\gamma|_{[\bar{t},1]}$ is again a geodesic, along which $f$ is convex.
Since $f(p)<f(y)<f(\gamma(t))$ for all $t \in [0,1)$,
$\bar{t}=1$ necessarily holds and we find that $y$ lies in the geodesic from $x$ to $p$.
Therefore, the proximal point algorithm goes straight towards the closest minimizer of $f$
(see Figure~\ref{fig:PPA}).

%%%%%%%%%%%%%%%%%%%%%%%%
\begin{figure}
\centering
\begin{picture}(400,150)

\put(50,40){\line(1,0){60}}
\put(75,5){\line(1,1){75}}
\put(150,80){\line(-1,1){55}}
\put(150,80){\line(1,0){110}}
\put(260,80){\line(1,2){30}}
\put(260,80){\line(2,1){80}}
\put(260,80){\line(1,-1){40}}
\put(300,40){\line(4,1){40}}
\put(300,40){\line(4,-1){40}}

\put(130,47){$p$}
\put(317,117){$x$}
\put(200,67){$y \in \sJ^f_{\tau}(x)$}
\put(288,84){$\gamma$}

\put(128,58){\rule{2pt}{2pt}}
\put(319,109){\rule{2pt}{2pt}}
\put(202,79){\rule{2pt}{2pt}}

\thicklines
\put(203,80){\line(1,0){57}}
\put(260,80){\line(2,1){59}}

\end{picture}
\caption{Proximal operator in $0$-hyperbolic spaces}\label{fig:PPA}
\end{figure}
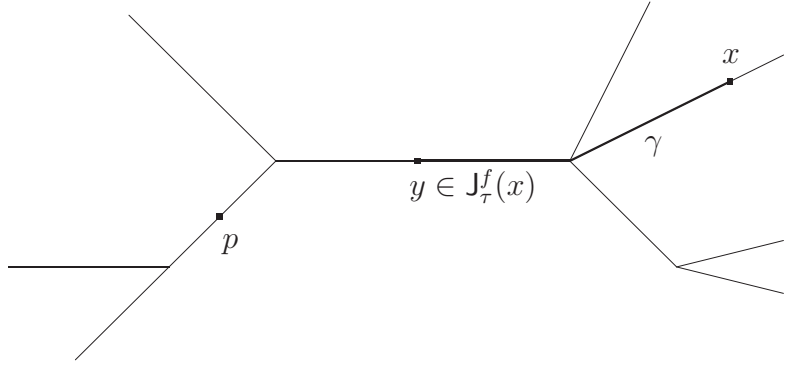
%%%%%%%%%%%%%%%%%%%%%%%%

The above argument is essentially indebted to the special property that
any (simple, constant speed) curve is a geodesic, however,
provides a rough picture of our strategy for general Gromov hyperbolic spaces
in the next section.

\section{Proofs of main results}\label{sc:proof}%%%%%
%%%%%%%%%%%%%%%%%%%%%%%%%%%%%%%

In this section, let $(X,d)$ be a proper $\delta$-hyperbolic geodesic space,
and $f:X \lra \R$ be a $K$-convex $L$-Lipschitz function with $K \ge 0$ and $L>0$.
Recall that $(X,d)$ is \emph{proper} if every bounded closed set is compact,
and $f$ is \emph{$L$-Lipschitz} if $|f(x)-f(y)| \le Ld(x,y)$ for all $x,y \in X$.
We also assume that $\inf_X f >-\infty$ and the infimum is attained at some point $p \in X$.
This is indeed the case if $K>0$ by a standard argument as follows
(see, e.g., \cite[Lemma~2.4.8]{AGSbook}).

\begin{lemma}\label{lm:K>0}
Let $(X,d)$ be a complete geodesic space
and $f$ be a lower semi-continuous $K$-convex function with $K>0$.
If $f$ is bounded below on some nonempty open set,
then $\inf_X f >-\infty$ and the infimum is attained at a unique point.
\end{lemma}

In the case of $K>0$, we also have the following a priori estimates in terms of $K$ and $L$.

\begin{remark}[A priori estimates]\label{rm:apri}
For any $x \in X$, we find
\[ f(p) +\frac{K}{2}d^2(p,x) \le f(x) \le f(p) +Ld(p,x), \]
where the first inequality follows from the $K$-convexity along a geodesic between $p$ and $x$.
Hence, we always have
\[ d(p,x) \le \frac{2L}{K}, \qquad f(x)-f(p) \le \frac{2L^2}{K}. \]
In particular, $\diam(X) \le 4L/K$.
\end{remark}

\subsection{Proof of Theorem~\ref{th:conv}}\label{ssc:thm1}%%%%%
%%%%%%%%%%%%%%%%%%%%%%%%%%%%%%%

We first prove Theorem~\ref{th:conv}.
The following proposition shows the first assertion \eqref{eq:conv}.
We remark that, in the current setting,
we have $\sJ^f_{\tau}(x) \neq \emptyset$ for any $x \in X$ and $\tau>0$.
In fact, the properness can be replaced with a weaker assumption that
every bounded closed set in each sublevel set $\{ y \in X \,|\, f(y) \le c \}$ is compact
(see \cite[Corollary~2.2.2, Lemma~2.4.8]{AGSbook}).

\begin{proposition}\label{pr:dpy}
Let $f:X \lra \R$ be $K$-convex and $L$-Lipschitz with $K \ge 0$ and $L>0$.
Then, for any $x \in X$, $\tau>0$, and $y \in \sJ^f_{\tau}(x)$, we have
\begin{equation}\label{eq:dpy}
d(p,y) \le d(p,x) -d(x,y) +\frac{4\sqrt{2\tau L\delta}}{\sqrt{K\tau +1}},
\end{equation}
where $p \in X$ is a minimizer of $f$.
\end{proposition}

The assertion \eqref{eq:dpy} can be rewritten as
\[ (x|p)_y \le \frac{2\sqrt{2\tau L\delta}}{\sqrt{K\tau +1}}. \]
In particular, if $\delta=0$, then $(x|p)_y=0$ holds and $y$ lies in a geodesic from $x$ to $p$
(recall \eqref{eq:dxg} and the discussion in Subsection~\ref{ssc:PPA}).

\begin{proof}
Assume $y \neq x$ without loss of generality.
On the one hand, for any geodesic $\gamma:[0,1] \lra X$ from $y$ to $x$,
we deduce from the choice of $y$ that
\[ f(y) +\frac{d^2(x,y)}{2\tau}
 \le f\big( \gamma(t) \big) +\frac{(1-t)^2 d^2(x,y)}{2\tau} \]
for all $t \in (0,1)$.
On the other hand, for some geodesic $\eta:[0,1] \lra X$ from $y$ to $p$,
the $K$-convexity implies
\[ f\big( \eta(s) \big) \le (1-s)f(y) +sf(p) -\frac{K}{2}(1-s)s d^2(p,y). \]
We set
\[ \bar{t} := \frac{(x|p)_y}{d(x,y)} \in [0,1], \qquad
 \bar{s} := \frac{\bar{t}d(x,y)}{d(p,y)} =\frac{(x|p)_y}{d(p,y)} \in [0,1]. \]
Then we have $d(y,\gamma(\bar{t}))=d(y,\eta(\bar{s}))=(x|p)_y$
and it follows from Lemma~\ref{lm:tripod} that
\[ d\big( \gamma(\bar{t}),\eta(\bar{s}) \big) \le 4\delta. \]
Hence, we find, since $f$ is $L$-Lipschitz,
\begin{align*}
(2\bar{t} -\bar{t}^2)\frac{d^2(x,y)}{2\tau}
&\le f\big( \gamma(\bar{t}) \big) -f(y)  \\
&\le f\big( \eta(\bar{s}) \big) -f(y) +4L\delta \\
&\le \bar{s} \big( f(p)-f(y) \big) -\frac{K}{2}(1-\bar{s})\bar{s} d^2(p,y) +4L\delta \\
&= \frac{\bar{t}d(x,y)}{d(p,y)}\big( f(p)-f(y) \big)
 -\frac{K}{2} \big( d(p,y) -\bar{t}d(x,y) \big) \bar{t} d(x,y) +4L\delta.
\end{align*}
Rearranging and multiplying the both sides with $2\tau/d^2(x,y)$ implies
\begin{equation}\label{eq:tbar}
(K\tau +1) \bar{t}^2
 -\bigg( \frac{2\tau}{d(x,y)} \frac{f(y)-f(p)}{d(p,y)} +K\tau \frac{d(p,y)}{d(x,y)} +2 \bigg) \bar{t}
 +\frac{8\tau L\delta}{d^2(x,y)} \ge 0.
\end{equation}

We regard the left hand side of \eqref{eq:tbar} as a quadratic polynomial of $\bar{t}$.
First, if the discriminant
\[ \Delta :=\bigg( \frac{\tau}{d(x,y)} \frac{f(y)-f(p)}{d(p,y)} +\frac{K\tau}{2} \frac{d(p,y)}{d(x,y)} +1 \bigg)^2
 -(K\tau +1) \frac{8\tau L\delta}{d^2(x,y)} \]
is negative, then we have
\[ K\tau d(p,y) +2d(x,y) < 4\sqrt{K\tau +1} \sqrt{2\tau L\delta} \]
since $f(y) \ge f(p)$.
Combining this with the triangle inequality and $d(x,y) \le d(p,x)$ from the choice of $y$ (and $f(y) \ge f(p)$),
we find
\begin{align*}
(K\tau +1)d(p,y)
&< 4\sqrt{K\tau +1} \sqrt{2\tau L\delta} -2d(x,y) +d(p,x)+d(x,y) \\
&\le 4\sqrt{K\tau +1} \sqrt{2\tau L\delta} +(K\tau +1) \big( d(p,x)-d(x,y) \big).
\end{align*}
This shows the claimed inequality \eqref{eq:dpy}.

Next, suppose $\Delta \ge 0$.
Observe that $\bar{t}$ lies left of the vertex of the polynomial, namely
\[ \bar{t} =\frac{(x|p)_y}{d(x,y)}
 \le \frac{1}{K\tau +1}
 \bigg( \frac{\tau}{d(x,y)} \frac{f(y)-f(p)}{d(p,y)} +\frac{K\tau}{2} \frac{d(p,y)}{d(x,y)} +1 \bigg) \]
holds, since
\begin{align*}
&2(K\tau +1)(x|p)_y -\bigg( 2\tau \frac{f(y)-f(p)}{d(p,y)} +K\tau d(p,y) +2d(x,y) \bigg) \\
&\le (K\tau -1)d(x,y) +d(p,y) -(K\tau +1)d(p,x) \\
&\le -d(x,y) +d(p,y) -d(p,x) \le 0.
\end{align*}
Thus, we obtain from \eqref{eq:tbar} that
\begin{align*}
(K\tau +1)\bar{t}
&\le \frac{\tau}{d(x,y)} \frac{f(y)-f(p)}{d(p,y)} +\frac{K\tau}{2} \frac{d(p,y)}{d(x,y)} +1 -\sqrt{\Delta} \\
&\le \sqrt{\bigg( \frac{\tau}{d(x,y)} \frac{f(y)-f(p)}{d(p,y)} +\frac{K\tau}{2} \frac{d(p,y)}{d(x,y)} +1 \bigg)^2
 -\Delta} \\
&= \sqrt{K\tau +1} \frac{2\sqrt{2\tau L\delta}}{d(x,y)}.
\end{align*}
Substituting $\bar{t}=(x|p)_y/d(x,y)$ yields \eqref{eq:dpy} and completes the proof.
%$\qedd$
\end{proof}

In the case of $K>0$ and $\tau >K^{-1}$,
we can estimate $d(x,y)$ in \eqref{eq:dpy} from below in terms of $K$ and $L$
(regardless of $\delta$) as follows.

\begin{lemma}\label{lm:dxy}
Let $f:X \lra \R$ be $K$-convex and $L$-Lipschitz with $K,L>0$.
Then we have, for any $x \in X$, $\tau>K^{-1}$, and $y \in \sJ^f_{\tau}(x)$,
\begin{equation}\label{eq:dxy}
d(x,y) \ge \bigg( 1-\frac{1}{K\tau} \bigg) \frac{f(x)-f(p)}{L}.
\end{equation}
\end{lemma}

\begin{proof}
On the one hand,
it follows from the choice of $y$ and the $L$-Lipschitz continuity that
\[ f(p) +\frac{d^2(p,x)}{2\tau} \ge f(y) +\frac{d^2(x,y)}{2\tau}
 \ge f(x) -Ld(x,y) +\frac{d^2(x,y)}{2\tau}. \]
On the other hand, the $K$-convexity implies (recall Remark~\ref{rm:apri})
\begin{equation}\label{eq:Kcon}
f(x) \ge f(p)+\frac{K}{2}d^2(p,x).
\end{equation}
Combining these furnishes
\begin{align*}
2\tau Ld(x,y) &\ge 2\tau Ld(x,y) -d^2(x,y) \ge 2\tau \big( f(x)-f(p) \big) -d^2(p,x) \\
&\ge \bigg( 2\tau-\frac{2}{K} \bigg) \big( f(x)-f(p) \big).
\end{align*}
%$\qedd$
\end{proof}

Now, plugging \eqref{eq:dxy} into \eqref{eq:dpy} completes
the proof of the second assertion \eqref{eq:conv-d}.

\begin{remark}\label{rm:conv}
In \eqref{eq:conv-d}, we have $d(p,y)<d(p,x)$ if
\[ f(x) > f(p) +\frac{4KL\tau \sqrt{2\tau L\delta}}{(K\tau -1)\sqrt{K\tau +1}}. \]
Note that this does not contradict the a priori bound $f(x)-f(p) \le 2L^2/K$
we mentioned in Remark~\ref{rm:apri}.
\end{remark}

\subsection{Proof of Corollary~\ref{cr:conv}}\label{ssc:cor}%%%%%
%%%%%%%%%%%%%%%%%%%%%%%%%%%%%%%

Let us first observe \eqref{eq:conv-f}, \eqref{eq:fxN} and \eqref{eq:dpxN}.
Combining \eqref{eq:dxy} with the choice of $y$,
we obtain \eqref{eq:conv-f} as
\[ f(y) \le f(x) -\frac{d^2(x,y)}{2\tau}
 \le f(x) -\frac{(K\tau -1)^2(f(x)-f(p))^2}{2(KL)^2 \tau^3}. \]
When we recursively choose $x_i \in \sJ^f_{\tau}(x_{i-1})$
for an arbitrary initial point $x_0 \in X$ and
\[ f(x_i) >f(p) +\frac{KL\tau\sqrt{2\tau}}{K\tau -1} \ve \]
holds for all $0 \le i \le N-1$, \eqref{eq:conv-f} yields
\[ f(x_N) < f(x_0) -N\ve^2. \]
Since $f(p) \le f(x_N)$, we find that $N <(f(x_0)-f(p))\ve^{-2}$ necessarily holds.
Therefore, we have \eqref{eq:fxN} for some $N <(f(x_0)-f(p))\ve^{-2}$.
Moreover, \eqref{eq:dpxN} follows from \eqref{eq:fxN} and \eqref{eq:Kcon}.

Turning to Corollary~\ref{cr:conv}, if
\[ d^2(p,x_i) > \frac{2L\tau}{K\tau -1} \bigg( \frac{4\sqrt{2\tau L\delta}}{\sqrt{K\tau +1}} +\ve^2 \bigg) \]
for all $0 \le i \le N-1$, then we deduce from \eqref{eq:Kcon} and \eqref{eq:conv-d} that
\[ d(p,x_N) < d(p,x_0) -N\ve^2. \]
Therefore, we have \eqref{eq:dpxN+} for some $N <d(p,x_0)\ve^{-2}$.

\subsection{Proof of Theorem~\ref{th:cont}}\label{ssc:thm2}%%%%%
%%%%%%%%%%%%%%%%%%%%%%%%%%%%%%%

We finally prove the contraction inequalities in Theorem~\ref{th:cont}.
The next lemma concerning convex functions on an interval is a well known fact.

\begin{lemma}\label{lm:1D}
Let $f:[0,\infty) \lra \R$ be a lower semi-continuous convex function
attaining its minimum at $0$.
Then, for any $\tau>0$ and $0<t_1<t_2$, we have
\[ 0 \le s_2-s_1 \le t_2-t_1, \]
where $s_i \in \sJ^f_{\tau}(t_i)$ for $i=1,2$.
\end{lemma}

\begin{proof}
We give a proof for thoroughness.
Note that, by hypotheses, $f$ is continuous and non-decreasing on $[0,\infty)$.
Thus, $s_i \le t_i$ holds.
Observe also that, for each $t>0$,
the function $s \longmapsto f(s)+(t-s)^2/(2\tau)$ is $(\tau^{-1})$-convex
and has a unique minimizer.
Hence, we have $\sJ^f_{\tau}(t_i)=\{s_i\}$.

We denote by $f'_+$ and $f'_-$ the right and left derivatives of $f$, respectively.
Since
\[ f'_-(s) -\frac{t_1-s}{\tau} > f'_-(s) -\frac{t_2-s}{\tau} >0 \]
for all $s>s_2$, we have $s_1 \le s_2$.
In particular, $s_2=0$ implies $s_1=0$.
Now, suppose $s_2>0$.
Then we have, by the choices of $s_1$ and $s_2$,
\[  f'_+(s_1) -\frac{t_1-s_1}{\tau} \ge 0, \qquad f'_-(s_2) -\frac{t_2-s_2}{\tau} \le 0. \]
Since $f'_+(s_1) \le f'_-(s_2)$ by the convexity of $f$,
we obtain $t_1-s_1 \le t_2 -s_2$.
%$\qedd$
\end{proof}

We are ready to prove Theorem~\ref{th:cont}.
Recall that $D=d(p,x_1) \vee d(p,x_2)$ and we assume $d(p,y_1) \le d(p,y_2)$.
Let $\gamma_i:[0,d(p,x_i)] \lra X$ be a unit speed geodesic from $p$ to $x_i$ (along which $f$ is $K$-convex),
and $\bar{y}_i$ be a point in $\gamma_i$ closest to $y_i$.
It follows from \eqref{eq:dxg} and Proposition~\ref{pr:dpy} that
\begin{equation}\label{eq:dyy}
d(y_i,\bar{y}_i) \le (x_i|p)_{y_i} +2\delta
 \le \frac{2\sqrt{2\tau L\delta}}{\sqrt{K\tau +1}} +2\delta =:C_1.
\end{equation}

If $d(p,y_1) \ge (x_1|x_2)_p$, then we have
\[ d(p,\bar{y}_1) \wedge d(p,\bar{y}_2) \ge d(p,y_1) \wedge d(p,y_2) -C_1 \ge (x_1|x_2)_p -C_1. \]
Hence, we obtain from Lemma~\ref{lm:trian}, \eqref{eq:dyy} and Proposition~\ref{pr:dpy} that
\begin{align*}
12\delta
&\ge 2(x_1|x_2)_p -2(\bar{y}_1|\bar{y}_2)_p -2C_1 \\
&\ge 2(x_1|x_2)_p -2(y_1|y_2)_p -6C_1 \\
&= d(y_1,y_2) -d(x_1,x_2) -2(x_1|p)_{y_1} -2(x_2|p)_{y_2} +d(x_1,y_1) +d(x_2,y_2) -6C_1 \\
&\ge d(y_1,y_2) -d(x_1,x_2) +d(x_1,y_1) +d(x_2,y_2) -\frac{8\sqrt{2\tau L\delta}}{\sqrt{K\tau +1}} -6C_1.
\end{align*}
In the case of $K>0$ and $\tau>K^{-1}$, Lemma~\ref{lm:dxy} further implies
\[ 12\delta
 \ge d(y_1,y_2) -d(x_1,x_2) +\bigg( 1-\frac{1}{K\tau} \bigg) \frac{f(x_1)+f(x_2)-2f(p)}{L}
 -\frac{8\sqrt{2\tau L\delta}}{\sqrt{K\tau +1}} -6C_1. \]
Thus, we have \eqref{eq:cont1} as well as \eqref{eq:cont1+}.

In the case of $d(p,y_1) <(x_1|x_2)_p$,
we shall essentially reduce to the $1$-dimensional situation (on $\gamma_2$)
and apply Lemma~\ref{lm:1D}.
We first consider ``projections'' to $\gamma_i$.
Take
\[ z_i \in \argmin_{z \in \gamma_i([0,d(p,x_i)])} \bigg\{ f(z) +\frac{d^2(x_i,z)}{2\tau} \bigg\}. \]
Since
\begin{align*}
f(z_i)+\frac{d^2(x_i,z_i)}{2\tau}
&\ge f(y_i)+\frac{d^2(x_i,y_i)}{2\tau} \\
&\ge f(\bar{y}_i) -Ld(y_i,\bar{y}_i) +\frac{d^2(x_i,\bar{y}_i)}{2\tau}
 -\frac{d(p,x_i)}{\tau} d(y_i,\bar{y}_i) \\
&\ge  f(\bar{y}_i)+\frac{d^2(x_i,\bar{y}_i)}{2\tau}
 -\bigg( L+\frac{D}{\tau} \bigg) C_1
\end{align*}
(we used in the second inequality the fact $d(x_i,y_i) \le d(p,x_i)$ from $y_i \in \sJ^f_{\tau}(x_i)$
as well as $d(x_i,\bar{y}_i) \le d(p,x_i)$ since $\bar{y}_i$ is on $\gamma_i$) and
\begin{equation}\label{eq:K+tau}
f(\bar{y}_i)+\frac{d^2(x_i,\bar{y}_i)}{2\tau}
 \ge f(z_i)+\frac{d^2(x_i,z_i)}{2\tau} +\frac{K+\tau^{-1}}{2}d^2(\bar{y}_i,z_i)
\end{equation}
by the $(K+\tau^{-1})$-convexity of $t \longmapsto f(\gamma_i(t)) +d^2(x_i,\gamma_i(t))/(2\tau)$,
we have
\begin{equation}\label{eq:dyz}
d^2(\bar{y}_i,z_i) \le \frac{2\tau}{K\tau +1}
 \bigg( L+\frac{D}{\tau} \bigg) C_1 =:C_2^2.
\end{equation}
Then, we put $\tilde{x}_1:=\gamma_1((x_1|x_2)_p)$ and take
\[ \tilde{z}_1 \in \argmin_{z \in \gamma_1([0,d(p,x_1)])} \bigg\{ f(z) +\frac{d^2(\tilde{x}_1,z)}{2\tau} \bigg\}. \]
Since $f \circ \gamma_1$ is non-decreasing, $\tilde{z}_1$ lies between $p$ and $\tilde{x}_1$.
Moreover, we have $d(p,\tilde{z}_1) \le d(p,z_1)$ by $s_1 \le s_2$ in Lemma~\ref{lm:1D}.

%%%%%%%%%%%%%%%%%%%%%%%%
\begin{figure}
\centering
\begin{picture}(400,150)

\qbezier(200,15)(200,90)(185,110)
\qbezier(185,110)(165,140)(100,140)
\qbezier(200,15)(200,90)(215,110)
\qbezier(215,110)(235,140)(300,140)

\put(197,4){$p$}
\put(87,138){$x_1$}
\put(305,138){$x_2$}
\put(137,144){$\gamma_1$}
\put(253,144){$\gamma_2$}
%%%%%
\put(170,97){$y_1$}
\put(238,112){$y_2$}

\put(194,103){$\bar{y}_1$}
\put(220,130){$\bar{y}_2$}
\put(183,84){$z_1$}
\put(232,135){$z_2$}

\put(183,118){$\tilde{x}_1$}
\put(207,118){$\tilde{x}_2$}
\put(184,64){$\tilde{z}_1$}
\put(207,64){$\tilde{z}_2$}

\put(210,80){$\tilde{y}_2$}
%%%%%
\put(182,95){\rule{2pt}{2pt}}
\put(233,116){\rule{2pt}{2pt}}

\put(190,97){\rule{2pt}{2pt}}
\put(228,123){\rule{2pt}{2pt}}
\put(192,91){\rule{2pt}{2pt}}
\put(235,127){\rule{2pt}{2pt}}

\put(183,110){\rule{2pt}{2pt}}
\put(215,110){\rule{2pt}{2pt}}
\put(196,71){\rule{2pt}{2pt}}
\put(202,71){\rule{2pt}{2pt}}

\put(203,77){\rule{2pt}{2pt}}
\end{picture}
\caption{The case of $d(p,y_1) <(x_1|x_2)_p$}\label{fig:cont}
\end{figure}
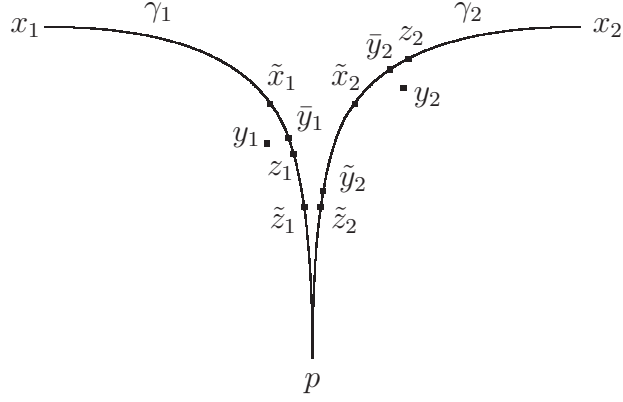
%%%%%%%%%%%%%%%%%%%%%%%%

Next, we further project from $\gamma_1$ to $\gamma_2$.
Precisely, we put $\tilde{x}_2:=\gamma_2((x_1|x_2)_p)$ and $\tilde{z}_2:=\gamma_2(d(p,\tilde{z}_1))$.
Then Lemma~\ref{lm:tripod} implies
\begin{equation}\label{eq:tilde}
d(\tilde{x}_1,\tilde{x}_2) \le 4\delta, \qquad d(\tilde{z}_1,\tilde{z}_2) \le 4\delta.
\end{equation}
Now we claim that
\begin{equation}\label{eq:dyz'}
d(\tilde{z}_1,y_2) \ge d(y_1,y_2) -8\delta -9C_1 -5C_2.
\end{equation}
Since $d(p,\tilde{z}_2) =d(p,\tilde{z}_1) \le d(p,z_1)$ and
\[ d(p,\bar{y}_2) \ge d(p,y_2) -C_1 \ge d(p,y_1) -C_1 \ge d(p,z_1) -2C_1 -C_2 \]
by \eqref{eq:dyy} and \eqref{eq:dyz}, we find
\begin{align*}
d\big( \gamma_2 \big( d(p,z_1) \wedge d(p,x_2) \big),\bar{y}_2 \big)
&= \big| d(p,\bar{y}_2) -d(p,z_1) \wedge d(p,x_2) \big| \\
&\le d(p,\bar{y}_2) -d(p,z_1) \wedge d(p,x_2) +4C_1+2C_2 \\
&\le d(p,\bar{y}_2) -d(p,\tilde{z}_2) +4C_1+2C_2 \\
&\le d(\tilde{z}_2,\bar{y}_2) +4C_1+2C_2.
\end{align*}
Moreover, it follows from $d(p,z_1) \le d(p,y_1)+C_1+C_2 <(x_1|x_2)_p +C_1+C_2$,
$(x_1|x_2)_p \le d(p,x_2)$ and Lemma~\ref{lm:tripod} that
\begin{align*}
&d\big( \gamma_2 \big( d(p,z_1) \wedge d(p,x_2) \big),z_1 \big) \\
&\le d\big( \gamma_2 \big( d(p,z_1) \wedge (x_1|x_2)_p \big),
 \gamma_1 \big( d(p,z_1) \wedge (x_1|x_2)_p \big) \big) +2C_1 +2C_2 \\
&\le 4\delta +2C_1 +2C_2.
\end{align*}
Together with \eqref{eq:dyy}, \eqref{eq:tilde} and \eqref{eq:dyz},
we can see the claim \eqref{eq:dyz'} as
\begin{align*}
d(\tilde{z}_1,y_2)
&\ge d(\tilde{z}_2,\bar{y}_2) -4\delta -C_1 \\
&\ge d\big( \gamma_2 \big( d(p,z_1) \wedge d(p,x_2) \big),\bar{y}_2 \big) -4\delta -5C_1 -2C_2 \\
&\ge d(z_1,\bar{y}_2) -d\big( \gamma_2 \big( d(p,z_1) \wedge d(p,x_2) \big),z_1 \big) -4\delta -5C_1 -2C_2 \\
&\ge d(z_1,\bar{y}_2) -8\delta -7C_1 -4C_2 \\
&\ge d(y_1,y_2) -8\delta -9C_1 -5C_2.
\end{align*}

We can also show that
\[ \tilde{y}_2 \in \argmin_{y \in \gamma_2([0,d(p,x_2)])} \bigg\{ f(y) +\frac{d^2(\tilde{x}_2,y)}{2\tau} \bigg\} \]
is close to $\tilde{z}_2$ in a similar way.
Namely, we observe from \eqref{eq:tilde}, $d(\gamma_1(d(p,\tilde{y}_2)),\tilde{y}_2) \le 4\delta$
from Lemma~\ref{lm:tripod}, and $d(p,\tilde{x}_1)=d(p,\tilde{x}_2)=(x_1|x_2)_p$ that
\begin{align*}
f(\tilde{y}_2) +\frac{d^2(\tilde{x}_2,\tilde{y}_2)}{2\tau}
&\ge f(\tilde{y}_2) +\frac{d^2(\tilde{x}_1,\tilde{y}_2)}{2\tau}
 -\frac{2d(\tilde{x}_2,\tilde{y}_2) +4\delta}{2\tau}4\delta \\
&\ge f(\tilde{y}_2) +\frac{d^2(\tilde{x}_1,\tilde{y}_2)}{2\tau}
 -\frac{d(p,\tilde{x}_2) +2\delta}{\tau}4\delta \\
&\ge f\big( \gamma_1 \big( d(p,\tilde{y}_2) \big) \big) +\frac{d^2(\tilde{x}_1,\gamma_1(d(p,\tilde{y}_2)))}{2\tau} \\
&\quad -\bigg( L+\frac{(x_1|x_2)_p+2\delta}{\tau} \bigg) 4\delta
 -\frac{(x_1|x_2)_p +2\delta}{\tau}4\delta.
\end{align*}
Then, by the choice of $\tilde{z}_1$, \eqref{eq:tilde}, $d(\tilde{x}_1,\tilde{z}_1)=d(\tilde{x}_2,\tilde{z}_2)$
and \eqref{eq:K+tau}, the right hand side is bounded from below by
\begin{align*}
&f(\tilde{z}_1) +\frac{d^2(\tilde{x}_1,\tilde{z}_1)}{2\tau}
 -\bigg( L+2\frac{(x_1|x_2)_p+2\delta}{\tau} \bigg) 4\delta \\
&\ge f(\tilde{z}_2) +\frac{d^2(\tilde{x}_2,\tilde{z}_2)}{2\tau}
 -\bigg( 2L+2\frac{(x_1|x_2)_p +2\delta}{\tau} \bigg) 4\delta \\
&\ge f(\tilde{y}_2) +\frac{d^2(\tilde{x}_2,\tilde{y}_2)}{2\tau}
 +\frac{K+\tau^{-1}}{2}d^2(\tilde{y}_2,\tilde{z}_2)
 -8\bigg( L+\frac{D+2\delta}{\tau} \bigg) \delta.
\end{align*}
This yields
\[ d^2(\tilde{y}_2,\tilde{z}_2)
 \le \frac{16\tau}{K\tau +1} \bigg( L+\frac{D+2\delta}{\tau} \bigg) \delta
 =:C_3^2. \]

Finally, we apply the $1$-dimensional contraction in Lemma~\ref{lm:1D}
to see $d(\tilde{y}_2,z_2) \le d(\tilde{x}_2,x_2)$.
Therefore, together with \eqref{eq:dyz'}, \eqref{eq:tilde}, \eqref{eq:dyy} and \eqref{eq:dyz},
we obtain
\begin{align*}
d(y_1,y_2)
&\le d(\tilde{z}_1,y_2) +8\delta +9C_1+5C_2 \\
&\le d(\tilde{z}_2,\bar{y}_2) +12\delta +10C_1+5C_2 \\
&\le d(\tilde{y}_2,z_2) +12\delta +10C_1+6C_2 +C_3 \\
&\le d(\tilde{x}_2,x_2) +12\delta +10C_1+6C_2 +C_3.
\end{align*}
Recalling $\tilde{x}_2=\gamma_2((x_1|x_2)_p)$, we observe
\[ d(\tilde{x}_2,x_2) =d(p,x_2) -(x_1|x_2)_p =d(x_1,x_2) -(p|x_2)_{x_1}. \]
This completes the proof of \eqref{eq:cont2} with
\[ C(K,L,D,\tau,\delta) =12\delta +10C_1+6C_2 +C_3. \]

\subsection{Further problems}\label{ssc:prob}%%%%%
%%%%%%%%%%%%%%%%%%%%%%%%%%%%%%%

We discuss some possible directions of further researches,
besides improvements of the estimates in Theorems~\ref{th:conv}, \ref{th:cont}
and Corollary~\ref{cr:conv}.

\begin{enumerate}[(A)]
\item
As we mentioned in Subsection~\ref{ssc:Ghyp},
the Gromov hyperbolicity makes sense for discrete spaces.
Therefore, it is interesting to consider some generalizations of the results in this article
to discrete Gromov hyperbolic spaces.
Then, it is a challenging problem to formulate and analyze
$K$-convex functions on discrete Gromov hyperbolic spaces
(possibly for some special classes such as hyperbolic groups).
We refer to \cite{Mu} for the theory of convex functions on $\mathbb{Z}^N$
(called \emph{discrete convex analysis}),
and to \cite{Hi1,Ko} for some generalizations to graphs and trees, respectively.

\item
It is also interesting to consider \emph{simulated annealing} in Gromov hyperbolic spaces,
namely proximal point algorithm with noise.
With this method it is expected that one can approximate a global minimizer
even for quasi-convex functions or $K$-convex functions with $K<0$.

\item
Related to the above problems,
it is worthwhile considering ``convex functions of large scale'',
preserved by \emph{quasi-isometries}.
This would provide a natural generalization of our research
since the Gromov hyperbolicity is preserved by quasi-isometries between geodesic spaces
(see, e.g., \cite[Theorem~3.18]{Va}).
\end{enumerate}
%\medskip

\textit{Acknowledgements.}
I would like to thank Hiroshi Hirai for his comments on convex functions on discrete spaces.
This work was supported in part by JSPS Grant-in-Aid for Scientific Research (KAKENHI) 19H01786, 22H04942.

{\small%%%

}

\end{document}